\documentclass{article}
 \pdfoutput=1
\usepackage{amsmath,amsfonts,amssymb,latexsym,graphics,epsfig}
\usepackage{color}
\usepackage{amsthm}
\usepackage{graphicx,url}
\usepackage{hyperref}
\usepackage[english]{babel}
\usepackage{epsfig}
\usepackage{enumerate}
\usepackage{graphicx}
\usepackage{floatrow}
\usepackage{float}
\usepackage{pifont}
\usepackage{caption}
\usepackage{subfig}
\usepackage{indentfirst}
\newtheorem{thm}{Theorem}[section]
\newtheorem{prp}[thm]{Proposition}
\newtheorem{lema}[thm]{Lemma}
\newtheorem{defi}[thm]{Definition}
\newtheorem{cor}[thm]{Corollary}
\newtheorem{claim}[thm]{Claim}

\newtheorem{example}[thm]{Example}
\newtheorem{remark}[thm]{Remark}
\newtheorem{property}[thm]{Property}

\begin{document}
\title{An application of Hoffman graphs for spectral characterizations of graphs}
\author{Qianqian Yang$^\dag$, Aida Abiad$^\ddag$, Jack H. Koolen$^\dag$\footnote{J.H.K. is partially supported by the National Natural Science Foundation of China (No.11471009)}
\\ \\
{\small $^\dag$School of Mathematical Sciences, University of Science and Technology of China} \\
{\small 96 Jinzhai, Hefei, 230026, Anhui, PR China}\\
{\small xuanxue@mail.ustc.edu.cn}\\
{\small $^\ddag$Dept. of Quantitative Economics, Maastricht University} \\
{\small  Maastricht, The Netherlands}\\
{\small A.AbiadMonge@maastrichtuniversity.nl} \\
{\small $^*$Wen-Tsun Wu Key Laboratory of CAS}\\
{\small 96 Jinzhai, Hefei, 230026, Anhui, PR China}\\
{\small koolen@ustc.edu.cn}\\
}
\maketitle
\date{}
\begin{abstract}
In this paper, we present the first application of Hoffman graphs for spectral characterizations of graphs. In particular, we show that the $2$-clique extension of the $(t+1)\times(t+1)$-grid is determined by its spectrum when $t$ is large enough. This result will help to show that the Grassmann graph $J_2(2D,D)$ is determined by its intersection numbers as a distance regular graph, if $D$ is large enough.\\[7pt]
{2010 Math. Subj. Class.:} 05C50, 05C75.
\end{abstract}
\section{Introduction}\label{intro}
Bang, Van Dam and Koolen \cite{BDK2008} showed that the Hamming graphs $H(3,q)$ are determined by their spectrum if $q \geq 36$. In this paper, we will show a similar result for the $2$-clique extension of the square grid. (For definitions we refer the reader to the next section.) In this paper we will show the following result:
\begin{thm}\label{maintheorem}
Let $G$ be a graph with spectrum
$$\big\{(4t+1)^{1},~(2t-1)^{2t},~(-1)^{(t+1)^{2}},~(-3)^{t^{2}}\big\}.$$
Then there exists a positive constant $C$ such that if $t \geq C$, then $G$ is the $2$-clique extension of the $(t+1) \times (t+1)$-grid.
\end{thm}
\begin{remark}
\item[$(i)$] The current estimates for $C$ are unrealistic high, since the proof implicitly uses Ramsey theory.
\item[$(ii)$] In \cite{ABH2015} it was shown that the $2$-coclique extension of the square grid is usually not determined by its spectrum.
\end{remark}

A motivation came from the study of Grassmann graphs. Gavrilyuk and Koolen studied \cite{GK2016} the question whether the Grassmann graph $J_2(2D, D)$ is determined as a distance-regular graph by its intersection numbers. (For definitions of distance-regular graphs and related notions we refer to \cite{BCN} and \cite{VKT}.) They showed that for any vertex, the subgraph induced by the neighbours of this vertex has the spectrum of the $2$-clique extension of a certain square grid. They used the main theorem of this paper to show that the Grassmann graph $J_2(2D, D)$ is determined as a distance-regular graph by its intersection numbers, if $D$ is large enough.

Another motivation for studying the $2$-clique extension of the $(t+1)\times(t+1)$-grid is because this is a connected regular graph with four distinct eigenvalues. Regular graphs with four distinct eigenvalues have been previously studied \cite{D1995}, and a key observation that we will use is that these graphs are walk-regular, which implies strong combinatorial information on the graph.

The starting point for our work is a result by Koolen et al. \cite{KYY}:
\begin{thm}\cite{KYY}\label{kyy}
There exists a positive integer $t$ such that any graph, that is cospectral with the $2$-clique extension of $(t_1 \times t_2)$-grid is the slim graph of a $2$-fat $\big\{$\raisebox{-0.5ex}{\includegraphics[scale=0.10]{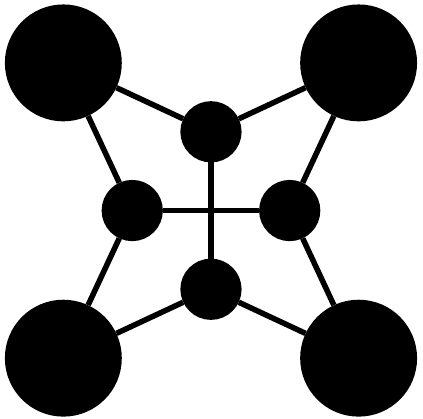},\includegraphics[scale=0.10]{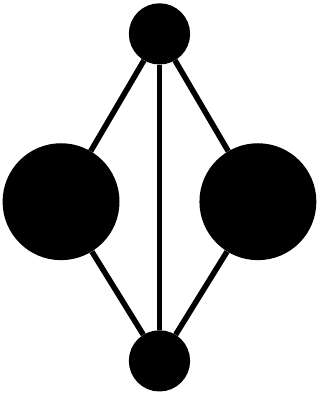},\includegraphics[scale=0.10]{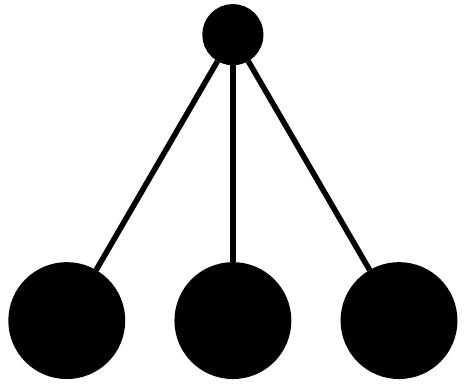}}$\big\}$-line Hoffman graph for all $t_1 \geq t_2 \geq t$.
\end{thm}
As a direct consequence, we obtain the following corollary:
\begin{cor}\label{corokyy}
There exists a positive integer $T$ such that any graph, that is cospectral with the $2$-clique extension of $(t+1) \times (t+1)$-grid is the slim graph of a $2$-fat $\big\{$\raisebox{-0.5ex}{\includegraphics[scale=0.10]{photo1},\includegraphics[scale=0.10]{photo2},\includegraphics[scale=0.10]{photo3}}$\big\}$-line Hoffman graph for all $t\geq T$.
\end{cor}
In view of Corollary \ref{corokyy}, our Theorem \ref{maintheorem} will follow from the following result:
\begin{thm}\label{maintheorem1intro}
Let $G$ be a graph cospectral with the $2$-clique extension of the $(t+1)\times (t+1)$-grid. If $G$ is the slim graph of a $2$-fat $\big\{$\raisebox{-0.5ex}{\includegraphics[scale=0.10]{photo1},\includegraphics[scale=0.10]{photo2},\includegraphics[scale=0.10]{photo3}}$\big\}$-line Hoffman graph, then $G$ is the $2$-clique extension of the $(t+1)\times (t+1)$-grid when $t>4$.
\end{thm}

The main focus of this paper is to prove Theorem \ref{maintheorem1intro}, and it is organized as follows. In Section \ref{preliminaries}, we review some preliminaries on graphs, interlacing and Hoffman graphs. Section \ref{sec:kyy} considers the graph cospectral with the $2$-clique extension of the $(t+1)\times(t+1)$-grid, which is the slim graph of the Hoffman graph having possible indecomposable factors isomorphic to the Hoffman graphs in Figure \ref{fg'}. In Section \ref{fordding}, we forbid two of the mentioned Hoffman graphs to occur as indecomposable factors. In Section \ref{determingorder}, the order of the quasi-cliques of the possible indecomposable factors is determined. Finally, in Section \ref{proof}, we finish the proof of Theorem \ref{maintheorem1intro}.

\section{Preliminaries}\label{preliminaries}
Throughout this paper we will consider only undirected graphs without loops or multiple edges. Suppose that $\Gamma$ is a graph with vertex set $V(\Gamma)$ with $|V(\Gamma)|=n$ and edge set $E(\Gamma)$. Let $A$ be the adjacency matrix of $\Gamma$, then the eigenvalues of $\Gamma$ are the eigenvalues of $A$. Let $\lambda_0,\lambda_1,\ldots,\lambda_t$ be the distinct eigenvalues of $\Gamma$ and $m_i$ be the multiplicity of $\lambda_i$ ($i=0,1,\ldots,t$). Then the multiset $\{\lambda_0^{m_0},\lambda_1^{m_1},\ldots,\lambda_t^{m_t}\}$ is called the \emph{spectrum} of $\Gamma$.

Two graphs are called \emph{cospectral} if they have the same spectrum.

For a vertex $x$, let $\Gamma_i(x)$ be the set of vertices at distance $i$ from $x$. When $i=1$, we also denote it by $N_\Gamma(x)$. For two distinct vertices $x$ and $y$, we denote the number of common neighbors between them by $\lambda_{x,y}$ if $x$ and $y$ are adjacent, and by $\mu_{x,y}$ if they are not.

Recall that the \emph{Kronecker product} $M_1\otimes M_2$ of two matrices $M_1$ and $M_2$ is obtained by replacing the $ij$-entry of $M_1$ by $({M_1})_{ij}M_2$ for all $i$ and $j$.  If $\tau$ and $\theta$ are eigenvalues of $M_1$ and $M_2$, then $\tau\theta$ is an eigenvalue of $M_1\otimes M_2$ \cite{Ch&G}.

Recall that a ($c$)-\emph{clique} (or complete graph) is a graph (on $c$ vertices) in which every pair of vertices is adjacent.

For an integer $q\geq 1$, the $q$-\emph{clique extension of} $\Gamma$ is the graph $\widetilde{\Gamma}$ obtained from $\Gamma$ by replacing each vertex $x\in V(\Gamma)$ by a clique $\widetilde{X}$ with $q$ vertices, such that $\tilde{x}\sim\tilde{y}$ (for $\tilde{x}\in\widetilde{X},~\tilde{y}\in\widetilde{Y},~\widetilde{X}\neq\widetilde{Y}$) in $\widetilde{\Gamma}$ if and only if $x\sim y$ in $\Gamma$.
If $\widetilde{\Gamma}$ is the $q$-clique extension of $\Gamma$, then $\widetilde{\Gamma}$ has adjacency matrix $J_q\otimes (A+I_n)-I_{qn}$, where $J_q$ is the all one matrix of size $q$ and $I_n$ is the identity matrix of size $n$.

In particular, if $q=2$ and $\Gamma$ has spectrum

 \begin{equation}\label{spectrum1}
 \big\{\lambda_0^{m_0},\lambda_1^{m_1},\ldots,\lambda_t^{m_t}\big\},
 \end{equation}
then it follows that the spectrum of $\widetilde{\Gamma}$ is
 \begin{equation}\label{spectrum2}
\big\{(2\lambda_0+1)^{m_0},(2\lambda_1+1)^{m_1},\ldots,(2\lambda_t+1)^{m_t},(-1)^{(m_0+m_1+\cdots+m_t)}\big\}.
 \end{equation}

In case that $\Gamma$ is a connected regular graph with valency $k$ and with adjacency matrix $A$ having exactly four distinct eigenvalues $\{\lambda_0=k,\lambda_1,\lambda_2,\lambda_3\}$, then $A$ satisfies the following (see for example \cite{H1963}):
\begin{equation}\label{HoffmanPolynomial}
A^{3}-\left(\displaystyle\sum_{i=1}^{3}\lambda_i\right)A^{2}+\left(\displaystyle\sum_{1\leq i<j\leq 3}\lambda_i\lambda_j\right)A-\lambda_1\lambda_2\lambda_3I=\frac{\prod_{i=1}^{3}(k-\lambda_i)}{n}J.
\end{equation}

We also need to introduce an important spectral tool that will be used throughout this paper: eigenvalue interlacing.
\begin{lema}\cite[Interlacing]{H1995}
Let $A$ be a real symmetric $n\times n$  matrix with
eigenvalues $\lambda_1\ge\cdots\ge \lambda_n$. For some $m<n$,
let $S$ be a real $n\times m$ matrix with orthonormal columns,
$S^{T}S=I$, and consider the matrix $B=S^{T}AS$,
with eigenvalues $\mu_1\ge\cdots\ge \mu_m$. Then,
\begin{itemize}
\item[$(i)$] the eigenvalues of $B$ interlace those of
    $A$, that is,
\begin{equation}
\lambda_i\ge \mu_i\ge \lambda_{n-m+i},\qquad i=1,\ldots, m.
\end{equation}
\item[$(ii)$] if there exists an integer $j\in \{1,2,\ldots,m\}$ such that $\lambda_{i}=\mu_{i}$ for $1\leq i\leq j$ and $\lambda_{n-m+i}=\mu_{i}$ for $j+1\leq i \leq m$, then the interlacing is tight and $SB=AS$.
\end{itemize}
\end{lema}

Two interesting particular cases of interlacing are obtained by choosing appropriately the matrix $S$. If $S=\left(
                                                                                                               \begin{array}{cc}
                                                                                                                 I & O \\
                                                                                                                 O & O \\
                                                                                                               \end{array}
                                                                                                             \right)$, then $B$ is just a principal submatrix of $A$. If $\pi=\{V_{1},\ldots,V_{m}\}$ is a partition of the
vertex set $V$, with each $V_{i}\neq \emptyset$, we can take for $\widetilde{B}$ the so-called quotient matrix of $A$ with respect to $\pi$.
Let $A$ be partitioned according to $\pi$:

$$A=\left( \begin{array}{ccc}
A_{1,1} & \cdots & A_{1,m} \\
\vdots &  & \vdots \\
A_{m,1} & \cdots & A_{m,m}
 \end{array} \right),$$

where $A_{i,j}$ denotes the submatrix (block) of $A$ formed by
rows in $V_{i}$ and columns in $V_{j}$. The \emph{characteristic matrix} $C$ is the $n\times m$ matrix whose $j^{\text{th}}$ column is the characteristic vector of $V_j$ ($j=1,\ldots,m$).

Then, the \emph{quotient matrix} of $A$ with respect to $\pi$ is the $m\times m$ matrix $\widetilde{B}$ whose entries are the average row sums of the blocks of $A$, more precisely:

\begin{equation*}
(\widetilde{B})_{i,j}=\frac{1}{|V_{i}|}(C^{T}AC)_{i,j}.
\end{equation*}

The partition $\pi$ is called \emph{equitable} (or \emph{regular}) if each block $A_{i,j}$ of $A$ has constant row (and column) sum, that is, $C\widetilde{B}=AC$.

\begin{lema}\label{quotienteigenvalues2}
Suppose $\widetilde{B}$ is the quotient matrix of a symmetric partitioned matrix $A$.
 \item[$(i)$] The eigenvalues of $\widetilde{B}$ interlace the eigenvalues of $A$.
 \item[$(ii)$] If the interlacing is tight, then the partition $\pi$ is equitable.
\end{lema}

\begin{lema}\cite[Theorem $9.3.3$]{Ch&G}\label{quotienteigenvalues1}
If $\pi$ is an equitable partition of a graph $\Gamma$, then the characteristic polynomial of $\widetilde{B}$ divides the characteristic polynomial of $A$.
\end{lema}

\subsection{Hoffman graphs}
We will need the following properties and definitions related to Hoffman graphs.

\begin{defi}
A \emph{Hoffman graph} $\mathfrak{h}$ is a pair $(H,\mu)$ of a graph $H=(V,E)$ and a labeling
map $\mu : V\rightarrow \{f,s\}$, satisfying the following conditions:
\begin{itemize}
\item[$(i)$] every vertex with label $f$ is adjacent to at least one vertex with label $s$;
\item[$(ii)$] vertices with label $f$ are pairwise non-adjacent.
\end{itemize}
We call a vertex with label $s$ a \emph{slim vertex}, and a vertex with label $f$ a \emph{fat vertex}. We denote
by $V_s=V_s(\mathfrak{h})$ (resp. $V_f(\mathfrak{h})$) the set of slim (resp. fat) vertices of $\mathfrak{h}$.

\vspace{0.1cm}
For a vertex $x$ of $\mathfrak{h}$, we define $N_{\mathfrak{h}}^{s}(x)$ (resp. $N_{\mathfrak{h}}^{f}(x)$) the set of slim (resp. fat) neighbors of $x$ in $\mathfrak{h}$.
If every slim vertex of a Hoffman graph $\mathfrak{h}$ has a fat neighbor, then we call $\mathfrak{h}$ \emph{fat}. And if every slim vertex has at least $t$ fat neighbors, we call $\mathfrak{h}$ $t$-\emph{fat}. In a similar fashion, we define $N^{f}(x_1,x_2)=N^{f}_\mathfrak{h}(x_1,x_2)$ to be the set of common fat neighbors of two slim vertices $x_1$ and $x_2$ in $\mathfrak{h}$ and $N^{s}(F_1,F_2)=N^{s}_\mathfrak{h}(F_1,F_2)$ to be the set of common slim neighbors of two fat vertices $F_1$ and $F_2$ in $\mathfrak{h}$.
\end{defi}

\vspace{0.1cm}
The \emph{slim graph} of a Hoffman graph $\mathfrak{h}$ is the subgraph of $H$ induced by $V_s(\mathfrak{h})$.

\vspace{0.1cm}
A Hoffman graph $\mathfrak{h_1 }= (H_1, \mu_1)$ is called an \emph{induced Hoffman subgraph} of $\mathfrak{h}=(H, \mu)$, if $H_1$ is an induced subgraph of $H$ and $\mu_1(x) = \mu(x)$ for all vertices $x$ of $H_1$.

\vspace{0.1cm}
Let $W$ be a subset of $V_s(\mathfrak{h})$. An \emph{induced Hoffman subgraph of $\mathfrak{h}$ generated by $W$}, denoted by $\langle W\rangle_{\mathfrak{h}}$, is the Hoffman subgraph of $\mathfrak{h}$ induced by $W \bigcup\{f\in V_f(\mathfrak{h})~|f \sim w \text{ for some }w\in W \}$.

\vspace{0.1cm}
A \emph{quasi-clique} is a subgraph of the slim graph of $\mathfrak{h}$ induced by the neighborhood of a fat vertex of $\mathfrak{h}$. If a quasi-clique is induced by the neighborhood of fat vertex $F$, we say it is the quasi-clique corresponding to $F$ and denote it by $Q_{\mathfrak{h}}(F)$.


\begin{defi} For a Hoffman graph $\mathfrak{h}=(H,\mu)$, let $A$ be the adjacency matrix of $H$
\begin{eqnarray*}
A=\left(
\begin{array}{cc}
A_s  & C\\
C^{T}  & O
\end{array}
\right)
\end{eqnarray*}

in a labeling in which the fat vertices come last. The \emph{special matrix} $\mathcal{S}(\mathfrak{h})$ of $\mathfrak{h}$ is the real symmetric matrix $\mathcal{S}(\mathfrak{h}):=A_s-CC^{T}.$ The eigenvalues of $\mathfrak{h}$ are the eigenvalues of $\mathcal{S}(\mathfrak{h})$.
\end{defi}

Note that $\mathfrak{h}$ is not determined by $\mathcal{S}$, since different $\mathfrak{h}$ may have the same special matrix $\mathcal{S}$.
Observe also that if there are not fat vertices, then $\mathcal{S}(\mathfrak{h})=A_s$ is just the standard adjacency matrix.

Now we introduce two key concepts in this work: the direct sum of Hoffman graphs and line Hoffman graphs.

\begin{defi}\textup{(Direct sum of Hoffman graphs)}\label{directsum}
Let $\mathfrak{h}$ be a Hoffman graph and $\mathfrak{h}^1$ and $\mathfrak{h}^2$ be two induced Hoffman subgraphs of $\mathfrak{h}$. The Hoffman graph $\mathfrak{h}$ is the \emph{direct sum} of $\mathfrak{h}^1$ and $\mathfrak{h}^2$, that is $\mathfrak{h}=\mathfrak{h}^1\bigoplus\mathfrak{h}^2$, if and only if $\mathfrak{h}^1, \mathfrak{h}^2$ and $\mathfrak{h}$ satisfy the following conditions:

\begin{itemize}
\item[$(i)$]$V(\mathfrak{h})=V(\mathfrak{h}^1)\bigcup V(\mathfrak{h}^2);$
\item[$(ii)$]$\big\{V_s(\mathfrak{h}^1),V_s(\mathfrak{h}^2)\big\}$ is a partition of $V_s(\mathfrak{h})$;
\item[$(iii)$]if $x\in V_s(\mathfrak{h}^i)$, $f\in V_f(\mathfrak{h})$ and $x\sim f$, then $f\in V_f(\mathfrak{h}^i)$;
\item[$(iv)$]if $x\in V_s(\mathfrak{h}^1)$ and $y\in V_s(\mathfrak{h}^2)$, then $x$ and $y$ have at most one common fat neighbor, and they have exactly one common fat neighbor if and only if they are adjacent.
\end{itemize}
\end{defi}

Let us show an example of how to construct a direct sum of two Hoffman graphs.
\begin{example}\label{example}
Let $\mathfrak{h}_1,\mathfrak{h}_2$ and $\mathfrak{h}_3$ be the Hoffman graphs represented in Figure \ref{1example}.
\begin{figure}[H]
\ffigbox{
\begin{subfloatrow}
\ffigbox[\FBwidth]{\caption*{$\mathfrak{h}_1$}}{\includegraphics[scale=0.8]{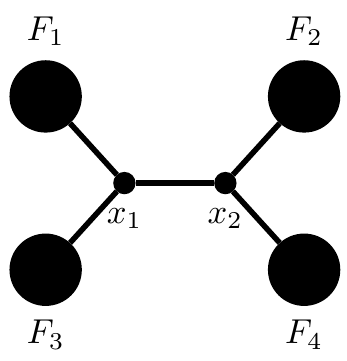}}
\ffigbox[\FBwidth]{\caption*{$\mathfrak{h}_2$}}{\includegraphics[scale=0.8]{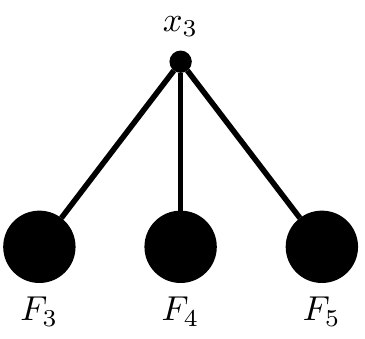}}
\ffigbox[1.2\FBwidth]{\caption*{$\mathfrak{h}_3$}}{\includegraphics[scale=0.8]{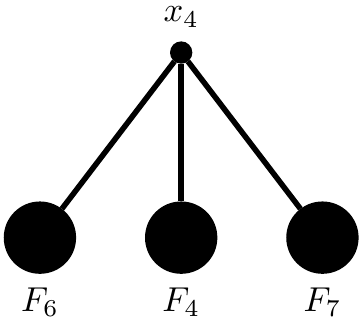}}
\end{subfloatrow}
}
{\caption{}\label{1example}}
\end{figure}
Then $\mathfrak{h}_1\bigoplus\mathfrak{h}_2,$ $\mathfrak{h}_1\bigoplus\mathfrak{h}_3$ and $\mathfrak{h}_2\bigoplus\mathfrak{h}_3$ are shown in Figure \ref{2example}.
\begin{figure}[H]
\ffigbox{
\begin{subfloatrow}
\ffigbox[1.0\FBwidth]{\caption*{$\mathfrak{h}_1\bigoplus\mathfrak{h}_2$}}{\includegraphics[scale=0.8]{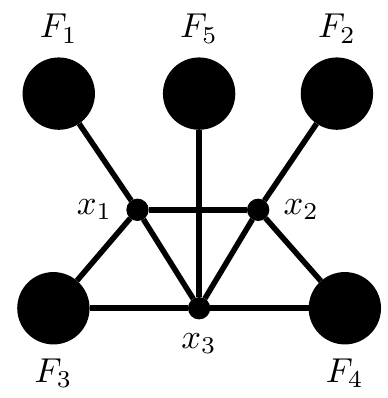}}
\ffigbox[1.0\FBwidth]{\caption*{$\mathfrak{h}_1\bigoplus\mathfrak{h}_3$}}{\includegraphics[scale=0.8]{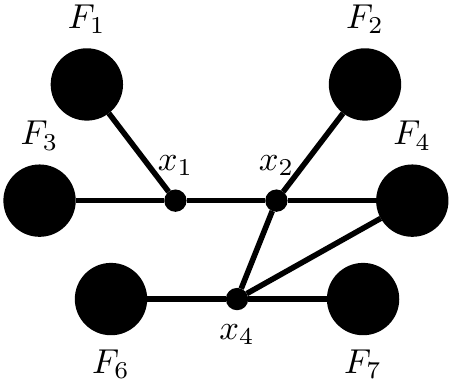}}
\ffigbox[1.2\FBwidth]{\caption*{$\mathfrak{h}_2\bigoplus\mathfrak{h}_3$}}{\includegraphics[scale=0.8]{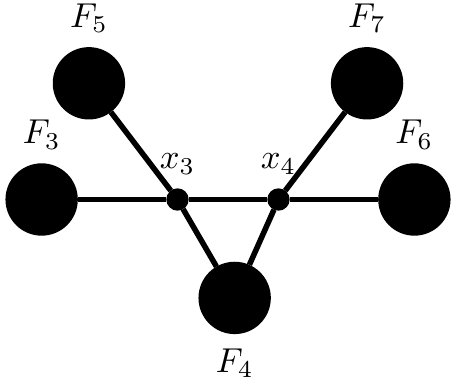}}
\end{subfloatrow}
}
{\caption{}\label{2example}}
\end{figure}
\end{example}
\begin{defi}
If a Hoffman graph $\mathfrak{h}$ is the direct sum of Hoffman graphs $\mathfrak{h}_1$ and $\mathfrak{h}^\prime$, then we call the Hoffman graph $\mathfrak{h}_1$ a \emph{factor} of $\mathfrak{h}$. If $\mathfrak{h}_1$ is indecomposable, then it is called indecomposable.
\end{defi}
\begin{defi}
Let $\mathfrak{G}$ be a family of Hoffman graphs. A Hoffman graph $\mathfrak{h}$ is called a $\mathfrak{G}$-\emph{line Hoffman graph} if $\mathfrak{h}$ is an induced Hoffman subgraph of Hoffman graph $\mathfrak{h}^\prime=\oplus_{i=1}^{r}\mathfrak{h}_i^\prime$ where $\mathfrak{h}_i^\prime$ is isomorphic to an induced Hoffman subgraph of some Hoffman graph in $\mathfrak{G}$ for $i=1,\dots,r$, such that $\mathfrak{h}^\prime$ has the same slim graph as $\mathfrak{h}$.
\end{defi}
%
%

\section{Cospectral graph with the $2$-clique extension of the $(t+1)\times(t+1)$-grid}\label{sec:kyy}
In this section, we study some consequences of Theorem \ref{kyy}. As mentioned in Section \ref{intro}, the main goal of this paper is to show Theorem \ref{maintheorem1intro}. Therefore, from now on we shall prepare the proof for Theorem \ref{maintheorem1intro}. \\
Let $t>0$ and for the rest of this paper, let $G$ be a graph cospectral with the $2$-clique extension of the $(t+1) \times (t+1)$-grid with adjacency matrix $A$. Since $G$ has the same spectrum as the $2$-clique extension of the $(t+1)\times (t+1)$-grid, $G$ is a regular graph with valency $k=4t+1$ and spectrum $$\big\{\eta_0^{m_0},\eta_1^{m_1},\eta_2^{m_2},\eta_3^{m_3}\big\}=\big\{(4t+1)^{1},~(2t-1)^{2t},~(-1)^{(t+1)^{2}},~(-3)^{t^{2}}\big\}.$$
Using (\ref{HoffmanPolynomial}) we obtain
\begin{equation*}
A^3+(5-2t)A^2+(7-8t)A+(3-6t)I=(16t+8)J.
\end{equation*}
Thus, we have
\begin{equation}\label{ways}
A^3_{(x,y)}= \left\{
\begin{array}{l l}
8t^2+4t, & \text{if } x=y;\\
24t+1-(5-2t)\lambda_{x,y}, & \text{if } x\sim y;\\
16t+8-(5-2t)\mu_{x,y}, & \text{if } x\not\sim y.
\end{array}
\right.
\end{equation}
If $G$ is the slim graph of a $2$-fat $\big\{$\raisebox{-0.5ex}{\includegraphics[scale=0.10]{photo1},\includegraphics[scale=0.10]{photo2},\includegraphics[scale=0.10]{photo3}}$\big\}$-line Hoffman graph, then there exists a $2$-fat Hoffman graph $\mathfrak{h}$, such that $\mathfrak{h}=\bigoplus_{i=1}^{s}\mathfrak{h}_i$ with slim graph $G$, and $\mathfrak{h}_i$ is isomorphic to one of the Hoffman graphs in the set $\mathfrak{G}=\big \{$\raisebox{-0.5ex}{\includegraphics[scale=0.10]{photo1},\includegraphics[scale=0.10]{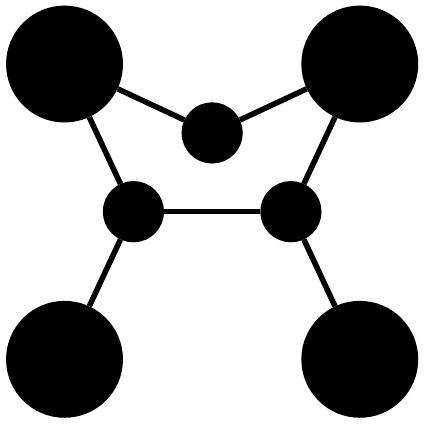},\includegraphics[scale=0.10]{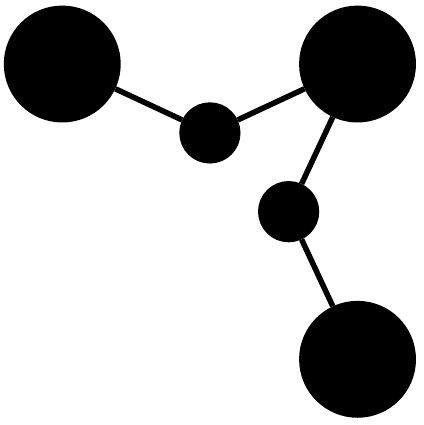},\includegraphics[scale=0.10]{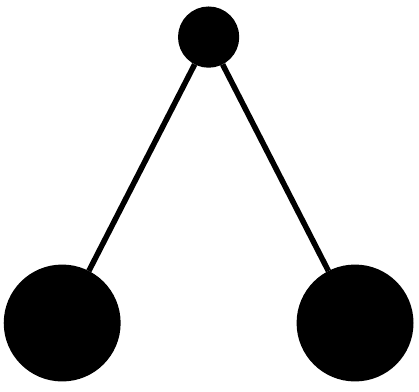},
\includegraphics[scale=0.10]{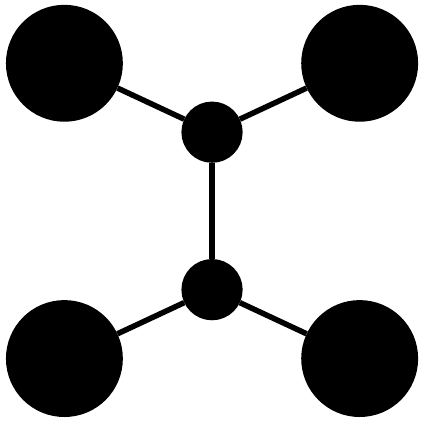},\includegraphics[scale=0.10]{photo2},\includegraphics[scale=0.10]{photo3}}$\big\}$ for $i=1,\ldots,s$.

We will now exclude two Hoffman graphs from the set $\mathfrak{G}$. To do so, we note the following remark:

\begin{remark}\label{remarkposition}
\item[$(i)$] The Hoffman graph \includegraphics[scale=0.10]{photo6} has the same slim graph as \includegraphics[scale=0.10]{photo3}.
\item[$(ii)$] The Hoffman graph \includegraphics[scale=0.10]{photo7} has the same slim graph as \includegraphics[scale=0.10]{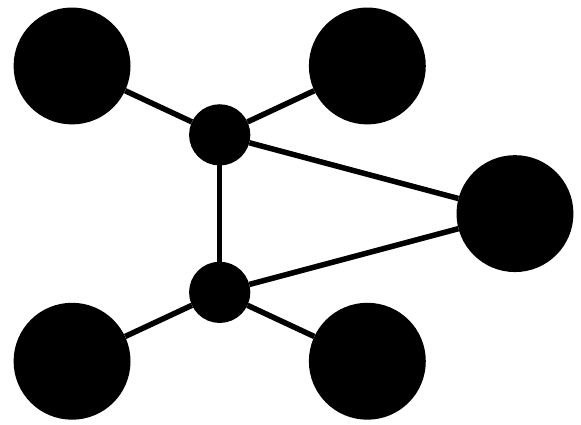}, which is the direct sum of two Hoffman graphs isomorphic to \includegraphics[scale=0.10]{photo3} with one common fat neighbor (see Example \ref{example}).
\end{remark}
Remark \ref{remarkposition} implies that we may assume that the $2$-fat Hoffman graph $\mathfrak{h}$, introduced before Remark \ref{remarkposition}, satisfies the following property, by adding fat vertices, if necessary.
\begin{property}\label{Hoffmanlinegraph}
\item[$(i)$] $\mathfrak{h}$ has $G$ as slim graph;
\item[$(ii)$] $\mathfrak{h}=\bigoplus_{i=1}^{s^\prime}\mathfrak{h}_i^\prime$, where $\mathfrak{h}_i^\prime$ isomorphic to one of the Hoffman graphs shown in Figure \ref{fg'}, for $i=1,\ldots,s^\prime$.
\end{property}
\begin{figure}[H]
\ffigbox{
\begin{subfloatrow}
\ffigbox[0.8\FBwidth]{\caption*{$\mathfrak{g}_1$}}{\includegraphics[scale=0.8]{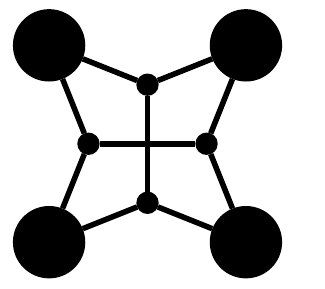}}
\ffigbox[1.2\FBwidth]{\caption*{$\mathfrak{g}_2$}}{\includegraphics[scale=0.8]{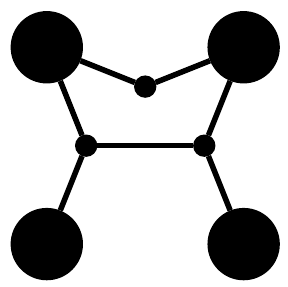}}
\ffigbox[0.8\FBwidth]{\caption*{$\mathfrak{g}_3$}}{\includegraphics[scale=0.8]{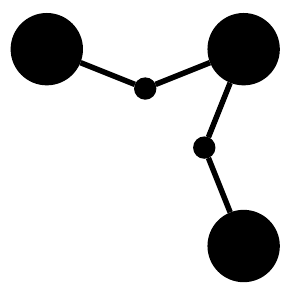}}
\ffigbox[0.8\FBwidth]{\caption*{$\mathfrak{g}_4$}}{\includegraphics[scale=0.8]{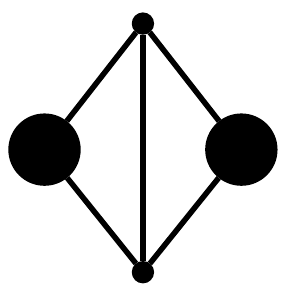}}
\ffigbox[0.8\FBwidth]{\caption*{$\mathfrak{g}_5$}}{\includegraphics[scale=0.8]{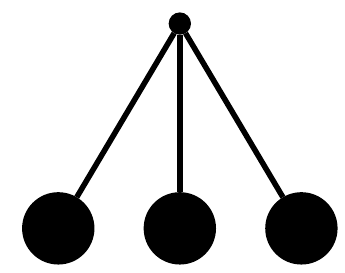}}
\end{subfloatrow}
}
{\caption{}\label{fg'}}
\end{figure}
Using Property \ref{Hoffmanlinegraph} and the definition of direct sum, we obtain the following lemma:
\begin{lema}\label{commonneighborsg'}
\item [($i$)] Any two distinct fat vertices $F_1$ and $F_2$ of $\mathfrak{h}$ have at most two common slim neighbors, i.e., $|N^s_{\mathfrak{h}}(F_1,F_2)|\leq 2$, and if $F_1$ and $F_2$ have exactly two common slim neighbors $x_1$ and $x_2$, then $x_1$ and $x_2$ are adjacent. In particular, this means that in this case, \raisebox{-3.5ex}{\includegraphics[scale=0.8]{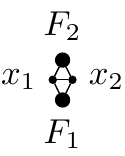}} is an indecomposable factor of $\mathfrak{h}$.
\item [($ii$)] If \raisebox{-2.5ex}{\includegraphics[scale=0.8]{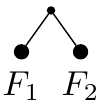}} is an induced Hoffman subgraph of one of the $\mathfrak{h_i}$ of Figure \ref{fg'}, and $\mathfrak{h_i}\not\backsimeq$ \raisebox{-1.2ex}{\includegraphics[scale=0.8]{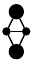}}, then $F_1$ and $F_2$ have exactly one slim common neighbor in $\mathfrak{h}$.
\end{lema}
\begin{proof} ($i$) Suppose that $N^s_{\mathfrak{h}}(F_1,F_2)=\{x_1,x_2,\ldots,x_p\}$. By Definition \ref{directsum} $(iv)$, we find that these $p$ distinct slim vertices and two fat vertices should be in the same indecomposable factor of $\mathfrak{h}$. By Figure \ref{fg'}, we see that if $p\geq 2$, then $p=2$ and the only indecomposable factor is (isomorphic to) $\mathfrak{g}_4$.

\vspace{0.2cm}
($ii$) This follows from ($i$).
\end{proof}

%

\section{Forbidding factors $\mathfrak{g_1}$ and $\mathfrak{g_2}$}\label{fordding}
In this section, we will show that $\mathfrak{g}_1$ and $\mathfrak{g}_2$ can not occur as an indecomposable factor of $\mathfrak{h}$. For this, we first need  the following lemma:

\begin{lema}\label{disjointcn}
Any two distinct nonadjacent vertices $x$ and $y$ in $G$ have at most $2t+2$ common neighbors, that is, $\mu_{x,y}\leq 2t+2$.
\end{lema}
\begin{proof}
Define a matrix $M$ as follows:
\begin{equation}\label{semidefinitematrix1}
\begin{split}
M&=(A-\eta_1I)(A-\eta_2I)\\
 &=(A-(2t-1)I)(A+I)\\
 &=A^{2}-2(t-1)A-(2t-1)I.
\end{split}
\end{equation}
Then $M$ is positive semidefinite (as $A$ has no eigenvalues  between $\eta_1$ and $\eta_2$), and we have
\begin{equation}\label{semidefinitematrix2}
M_{(x,y)}= \left\{
\begin{array}{l l}
k-(2t-1)=2t+2, & \text{if }x=y;\\
-2(t-1)+\lambda_{x,y},& \text{if }x\sim y;\\
\mu_{x,y}, & \text{if }x\not\sim y.
\end{array}
\right.
\end{equation}

Since $M$ is positive semidefinite, all its principal submatrices are positive semidefinite. Let $x$ and $y$ be two distinct nonadjacent vertices of $G$. Then  $$\left(
\begin{array}{cc}
2t+2  & \mu_{x,y}\\
\mu_{x,y} & 2t+2
\end{array}
\right)$$
 is positive semidefinite and hence $\mu_{x,y}\le 2t+2$ holds.
\end{proof}

Using Lemma \ref{disjointcn}, we obtain the following result:
\begin{lema}\label{g1'g2'}

\item[$(i)$] The Hoffman graph $\mathfrak{g}_1$ can not be an indecomposable factor of $\mathfrak{h}$ when $t>1$.
\item[$(ii)$] The Hoffman graph $\mathfrak{g}_2$ can not be an indecomposable factor of $\mathfrak{h}$ when $t>1$.
\end{lema}
\begin{proof}
$(i)$ Suppose that $\mathfrak{g}_1$ is an indecomposable factor of $\mathfrak{h}$, where $a_i=|V(Q_{\mathfrak{h}}(F_i))|$, as shown in Figure \ref{fg1'size}.
\begin{figure}[H]
\center
  \includegraphics[scale=1]{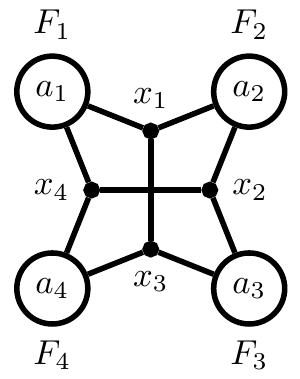}
  \vspace{0.1cm}
  \caption{$\mathfrak{g}_1$}
  \label{fg1'size}
\end{figure}
 By Lemma \ref{commonneighborsg'}, we find that $N^s_{\mathfrak{h}}(F_1,F_2)=\{x_1\},N^s_{\mathfrak{h}}(F_2,F_3)=\{x_2\},~N^s_{\mathfrak{h}}(F_1,F_4)=\{x_4\},~|N^s_{\mathfrak{h}}(F_1,F_3)|\leq 2,$ and $|N^s_{\mathfrak{h}}(F_2,F_4)|\leq 2.$ By the definition of direct sum, we know that if a vertex $x$ ($x\neq x_3$) is adjacent to $x_1$ in $G$, then $x\in N^s_{\mathfrak{h}}(F_1)$ or $x\in N^s_{\mathfrak{h}}(F_2)$. So $a_1+a_2-3=a_1-2+a_2-2+1=|N_G(x_1)|=k=4t+1$, that is $a_1+a_2=4t+4$. (In a similar way, we obtain that $a_2+a_3=a_3+a_4=a_1+a_4=4t+4$, so $a_1=a_3$ and $a_2=a_4$.) Note that $\mu_{x_1,x_2}=a_2-2+|N^s_{\mathfrak{h}}(F_1,F_3)|,~\mu_{x_1,x_4}=a_1-2+|N^s_{\mathfrak{h}}(F_2,F_4)|,$
$\lambda_{x_2,x_4}=|N^s_{\mathfrak{h}}(F_1,F_3)|+|N^s_{\mathfrak{h}}(F_2,F_4)|.$ We obtain
 \begin{equation}\label{1x1x2x4}
   \mu_{x_1,x_2}+\mu_{x_1,x_4}=4t+\lambda_{x_2,x_4}.
 \end{equation}

Without loss of generality, we may assume that $\mu_{x_1,x_4}\leq\mu_{x_1,x_2}$. From Lemma \ref{disjointcn} and Equation (\ref{1x1x2x4}), we obtain
\begin{equation}\label{newineqg1}
\begin{aligned}
&0\leq \lambda_{x_2,x_4}\leq 4, \quad \mu_{x_1,x_4}\le\mu_{x_1,x_2}\leq 2t+2,\\
&\text{and }2t-2+\lambda_{x_2,x_4}\leq\mu_{x_1,x_4}\leq 2t+\bigg\lfloor\frac{\lambda_{x_2,x_4}}{2}\bigg\rfloor.
\end{aligned}
\end{equation}

Take the positive semidefinite principal submatrix $M_1$ of $M$, corresponding to the vertices $\{x_1,x_2,x_4\}$. Then, we obtain (by using (\ref{semidefinitematrix2})):
\begin{equation*}
M_1=\begin{pmatrix}
2t+2  &\mu_{x_{1},x_{2}} &\mu_{x_{1},x_{4}}  \\
\mu_{x_{1},x_{2}}& 2t+2 &-2(t-1)+\lambda_{x_2,x_4} \\
\mu_{x_{1},x_{4}}&-2(t-1)+\lambda_{x_2,x_4} & 2t+2
\end{pmatrix}.
\end{equation*}

Replacing $\mu_{x_{1},x_{4}}$ by $\mu$ and $\lambda_{x_2,x_4}$ by $\lambda$ and using (\ref{1x1x2x4}), we have
\begin{equation*}
M_1=\begin{pmatrix}
2t+2  &4t+\lambda-\mu &\mu \\
4t+\lambda-\mu & 2t+2 &-2(t-1)+\lambda \\
\mu & -2(t-1)+\lambda & 2t+2
\end{pmatrix}.
\end{equation*}

The above matrix $M_1$ has determinant

\begin{equation*}
\begin{split}
 det(M_1)=&-32 t^3-8\lambda t^2+\big((8\lambda+32)\mu-4\lambda^2-16\lambda+32\big)t\\
 &-(2\lambda+8)\mu^2+(2\lambda^2+8\lambda)\mu-4\lambda^2-8\lambda,
\end{split}
\end{equation*}
where $0\leq \lambda\leq4,~2t-2+\lambda\leq\mu\leq 2t+\lfloor\frac{\lambda}{2}\rfloor$ (by (\ref{newineqg1})).

If $t>1$, by checking all the possible values of $\lambda$ and $\mu$, we obtain that $\text{det}(M_1)<0$ and this is impossible since $M_1$ is positive semidefinite.\\

$(ii)$ can be shown in a similar way. Suppose that $\mathfrak{g}_2$ is an indecomposable factor of $\mathfrak{h}$, where $a_i=|V(Q_{\mathfrak{h}}(F_i))|$, as shown in Figure \ref{fg2'size}.
\begin{figure}[H]
\center
  \includegraphics[scale=1]{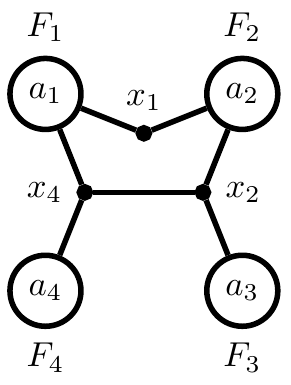}
  \vspace{0.1cm}
  \caption{$\mathfrak{g}_2$}
  \label{fg2'size}
\end{figure}

Then the submatrix $M_1$ is replaced by
\begin{equation*}
M_2=\begin{pmatrix}
2t+2  &4t+1+\lambda-\mu &\mu  \\
4t+1+\lambda-\mu& 2t+2 &-2(t-1)+\lambda \\
\mu&-2(t-1)+\lambda & 2t+2
\end{pmatrix},
\end{equation*}
with determinant:

\begin{equation*}
\begin{split}
 det(M_2)=&-32 t^3-(8\lambda+16)t^2+\big((8\lambda+32)\mu-4\lambda^2-20\lambda+14\big)t\\
 &-(2\lambda+8)\mu^2+(2\lambda^2+10\lambda+8)\mu-4\lambda^2-12\lambda-2,
\end{split}
\end{equation*}
where $0\leq \lambda=\lambda_{x_2,x_4}\leq4$ and $2t-1+\lambda\leq \mu=\mu_{x_{1},x_{4}}\leq 2t+\lfloor\frac{1+\lambda}{2}\rfloor$.

If $t>1$, by checking all the possible values of $\lambda$ and $\mu$, we obtain that $\text{det}(M_2)<0$ and the result follows, as this gives a contradiction.
\end{proof}

\section{The order of quasi-cliques}\label{determingorder}
\subsection{An upper bound on the order of quasi-cliques}
From the above section, we find that the only possible indecomposable factors of $\mathfrak{h}$ are $\mathfrak{g}_3$, $\mathfrak{g}_4$ and $\mathfrak{g}_5$.
\begin{prp}\label{quasicliqueorder}
Let $q$ be the order of a quasi-clique $Q$ corresponding to a fat vertex $F$ in $\mathfrak{h}$. Then $q\leq2t+2$ when $t>1$.
\end{prp}
\begin{proof}
We show the following three claims from which the proposition follows.

\begin{claim}\label{claimvalency}
In the quasi-clique $Q$, every vertex has valency at least $q-2$.
\end{claim}
\begin{proof}
If there exists a vertex that has two nonneighbors in $Q$, then in $\mathfrak{h}$, these three slim vertices should be in the same indecomposable factor by Definition \ref{directsum} $(iv)$. But neither \raisebox{-0.5ex}{\includegraphics[scale=0.1]{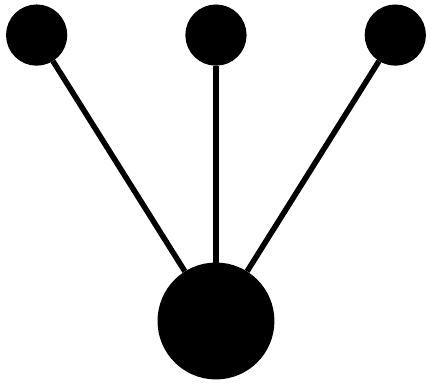}} nor \raisebox{-0.5ex}{\includegraphics[scale=0.1]{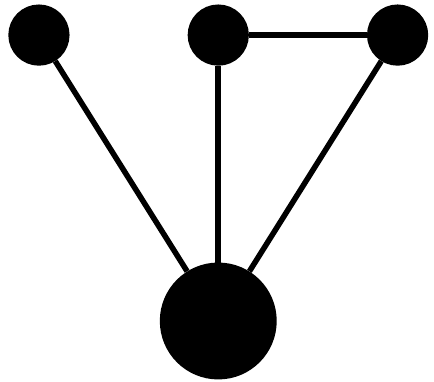}} is an induced Hoffman subgraph of $\mathfrak{g}_3$, $\mathfrak{g}_4$ or $\mathfrak{g}_5$. Hence the claim holds.
\end{proof}

\begin{claim}\label{claimbound}
The order $q$ of the quasi-clique $Q$ is at most $2t+3$ when $t>1$, and if $q=2t+3$, then $Q$ has exactly a vertex of valency $2t+2$.
\end{claim}
\begin{proof}
Let us consider the partition $\pi=\big\{V(Q),V(G)-V(Q)\big\}$ of $V(G)$. The quotient matrix $\widetilde{B}$ of $A$ with respect to the partition $\pi$ is
\begin{equation}\label{g2matrix}
  \widetilde{B}=\left(
     \begin{array}{cc}
       q-2+\epsilon & 4t+1-(q-2+\epsilon) \\
       \frac{(4t+1-(q-2+\epsilon))q}{2(t+1)^2-q} & 4t+1- \frac{(4t+1-(q-2+\epsilon))q}{2(t+1)^2-q} \\
     \end{array}
   \right)
\end{equation}
with eigenvalues $k(=4t+1)$ and $q-2+\epsilon- \frac{(4t+1-(q-2+\epsilon))q}{2(t+1)^2-q}$, where $0\leq\epsilon\leq1$ (by Claim \ref{claimvalency}). By interlacing (Lemma \ref{quotienteigenvalues2} $(i)$), we obtain that, the second eigenvalue of the quotient matrix $\widetilde{B}$ is at most $2t-1$, hence $q-2+\epsilon- \frac{(4t+1-(q-2+\epsilon))q}{2(t+1)^2-q}\leq 2t-1$ holds.

If $q=2t+4$, then
$2t+2+\epsilon-\frac{(2t-1-\epsilon)(t+2)}{t^2+t-1}=q-2+\epsilon- \frac{(4t+1-(q-2+\epsilon))q}{2(t+1)^2-q}\leq 2t-1$.
But this is not possible when $t>1$.

If $q=2t+3$, then (\ref{g2matrix}) becomes
\begin{equation*}
  \widetilde{B}=\left(
     \begin{array}{cc}
       2t+1+\epsilon & 2t-\epsilon \\
       \frac{(2t-\epsilon)(2t+3)}{2t^2+2t-1} & 4t+1-\frac{(2t-\epsilon)(2t+3)}{2t^2+2t-1}  \\
     \end{array}
   \right)
\end{equation*}
and $2t+1+\epsilon-\frac{(2t-\epsilon)(2t+3)}{2t^2+2t-1}=q-2+\epsilon-\frac{(4t+1-(q-2+\epsilon))q}{2(t+1)^2-q}\leq 2t-1$.
By solving this inequality, we have $0\leq\epsilon\leq\frac{1}{1+t}$. Suppose that there are $m_1$ vertices with valency $2t+1$ and $m_2$ vertices with valency $2t+2$ in $Q$. Then $$m_1+m_2=2t+3,$$
$$\frac{(2t+1)m_1+(2t+2)m_2}{m_1+m_2}=2t+1+\frac{m_2}{m_1+m_2}\leq 2t+1+\frac{1}{1+t}.$$
Since $m_1$ is an even number by the handshaking lemma, it follows that the only possible solution is $m_1=2t+2,~m_2=1$. So the claim holds.
\end{proof}

Finally, we show the following:
\begin{claim}\label{no2t+3}
There are no quasi-cliques of order $2t+3$ when $t>1$.
\end{claim}
\begin{proof}Assume that there exists a quasi-clique $Q^\prime$ with order $2t+3$, corresponding to fat vertex $F$ in $\mathfrak{h}$. Then, from Claim \ref{claimbound}, we obtain that, in $Q^\prime$, there exist two distinct vertices which are not adjacent, say $x_1$ and $x_2$. Now consider the factor containing the slim vertices $x_1,~x_2$ and fat vertex $F$. Then we see that $F$ should be the fat vertex $F_2$ in $\mathfrak{g}_3$ (in Figure \ref{fg3'size}) and $Q^\prime=Q_{\mathfrak{h}}(F_2)$ with order $a_2=2t+3$.

Moreover, we obtain that $a_1-1+a_2-2=|N_G(x_1)|=k=4t+1$ and $a_2-2+a_3-1=|N_G(x_2)|=k=4t+1$, where $a_1=|V(Q_{\mathfrak{h}}(F_1))|$ and $a_3=|V(Q_{\mathfrak{h}}(F_3))|$. Then $|V(Q_{\mathfrak{h}}(F_1))|=|V(Q_{\mathfrak{h}}(F_3))|=2t+1$ and $ V(Q_{\mathfrak{h}}(F_1))\bigcap V(Q_{\mathfrak{h}}(F_2))=\{x_1\}$, $V(Q_{\mathfrak{h}}(F_3))\bigcap V(Q_{\mathfrak{h}}(F_2))=\{x_2\}$ by Lemma \ref{commonneighborsg'} $(ii)$. By using Claim \ref{claimbound} again, we find that there are exactly $t+1$ pairs of non adjacent vertices in $Q^\prime$ and for each such a pair of vertices, we find two quasi-cliques with order $2t+1$ containing exactly one of them as a vertex, respectively. It means that there are at least $2t+2$ distinct quasi-cliques with order $2t+1$.\\
Now let us estimate the cardinality of the set $W=\big\{(x,Q^{\prime\prime})
~|~x\in V(Q^{\prime\prime})$ and $Q^{\prime\prime}$ is a quasi-clique corresponding to some fat vertex in $\mathfrak{h}$ with order $2t+1$ or $2t+3\big\}$ by double counting. On the one hand, $|W|\leq 2\cdot2(t+1)^2$,  since every vertex can only be a member of at most $2$ such quasi-cliques considering its valency is $4t+1$. On the other hand, we know there are at least $2t+2$ quasi-cliques with order $2t+1$ and at least $1$ quasi-clique with order $2t+3$. So $|W|\geq (2t+2)(2t+1)+1\cdot(2t+3)$. Hence $4(t+1)^2\geq (2t+2)(2t+1)+(2t+3)$, a contradiction. This shows the claim.
\end{proof}
The proposition follows from Claims \ref{claimvalency}, \ref{claimbound} and \ref{no2t+3}.
\end{proof}
In addition, we will give a lemma about cliques of $G$ that will be used in next sections.

\begin{lema}\label{HoffmanBound}
Let $c$ be the order of a clique $C$ in $G$, then $c\leq 2t+2$. If equality holds, then every vertex $x\in V(G)-V(C)$ has exactly $2$ neighbors in $C$.
\end{lema}
\begin{proof}
For the inequality case, exactly the same argument applies by replacing $\epsilon$ by $1$ in the proof of Claim \ref{claimbound}. If equality holds, then we have tight interlacing, since $\widetilde{B}$ has $k$ and $2t-1$ as its eigenvalues, which are also the largest and second largest eigenvalues of $A$. So by Lemma \ref{quotienteigenvalues2} $(ii)$, the partition $\pi=\big\{V(C),V(G)-V(C)\big\}$ is equitable and by (\ref{g2matrix}) ($q=2t+2$), we obtain that every vertex in $V(G)-V(C)$ has exactly $2$ neighbors in $C$.
\end{proof}


\subsection{Determining the order of the quasi-cliques for $\mathfrak{g}_3$ and $\mathfrak{g}_4$}
In this subsection, we will determine the order of quasi-cliques for each of the remaining indecomposable factors $\mathfrak{g}_3$ and $\mathfrak{g}_4$. First we consider the indecomposable factor $\mathfrak{g}_3$.
\begin{lema}\label{g3'}
Suppose that $\mathfrak{g}_3$ is an indecomposable factor of $\mathfrak{h}$ with fat vertices $F_1,F_2$ and $F_3$. Then for $i=1,2,3$, the quasi-clique $Q_{\mathfrak{h}}(F_i)$ corresponding to $F_i$ has order $2t+2$ when $t>1$.
\end{lema}
\begin{proof}
Let $\mathfrak{g}_3$ be an indecomposable factor of $\mathfrak{h}$ as shown in Figure \ref{fg3'size}, where $a_i=|V(Q_{\mathfrak{h}}(F_i))|$, for $i=1,2,3$.
\begin{figure}[H]
\center
  \includegraphics[scale=1]{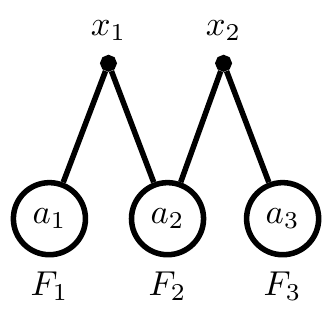}
  \vspace{0.1cm}
  \caption{$\mathfrak{g}_3$}
  \label{fg3'size}
\end{figure}

It is clear that $a_1-1+a_2-2=|N_G(x_1)|=k=4t+1$, that is, $a_1+a_2=4t+4$. From Proposition \ref{quasicliqueorder}, it follows that $a_1=a_2=2t+2$. By interchanging the roles of $x_1$ and $x_2$, the result follows.
\end{proof}

Now we consider the indecomposable factor $\mathfrak{g}_4$.
\begin{lema}\label{g4'}
Suppose that $\mathfrak{g}_4$ is an indecomposable factor of $\mathfrak{h}$ with fat vertices $K_1$ and $K_2$ and slim vertices $x$ and $y$. Then for $i=1,2$, the quasi-clique $Q_{\mathfrak{h}}(K_i)$ corresponding to $K_i$ has order $2t+2$ when $t>1$.

Moreover, the partition $\pi=\{V_1,V_2,V_3\}$ on $V(G)$ is equitable with quotient matrix $$\left(
\begin{array}{ccc}
1  & 4t & 0\\
2 & 2t-1 & 2t\\
0 & 4 & 2t-3
\end{array}
\right),$$ where $V_1=\{x,y\},~V_2=V(Q_{\mathfrak{h}}(K_1))\bigcup V(Q_{\mathfrak{h}}(K_2))-V_1$ and $V_3=V(G)-V_1\bigcup V_2.$
\end{lema}

\begin{proof}
Consider $\mathfrak{g}_1$ in Figure \ref{fg4'size}, where $d_i$ is the order of quasi-clique $Q_{\mathfrak{h}}(K_i)$, for $i=1,2$.
\begin{figure}[H]
\center
  \includegraphics[scale=1]{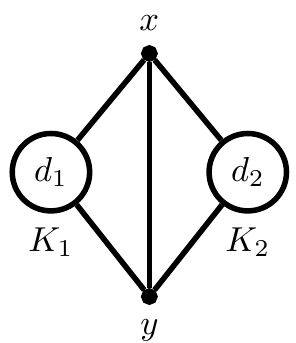}
  \vspace{0.1cm}
  \caption{$\mathfrak{g}_4$}
  \label{fg4'size}
\end{figure}
Then by definition of direct sum and Lemma \ref{commonneighborsg'} $(i)$, we obtain that $d_1-2+d_2-2+1=|N_G(x)|=4t+1$, that is, $d_1+d_2=4t+4$. By using Proposition \ref{quasicliqueorder} again, it is easy to see that $d_1=d_2=2t+2$.

%

Now we will show that the partition is equitable. Suppose that $\alpha$ is the average number of edges leading from a vertex in $V_3$ to vertices in $V_2$. Then the quotient matrix $\widetilde{B}$ of $A$ with respect to $\pi$ is:
\begin{eqnarray*}
\widetilde{B}=\left(
\begin{array}{ccc}
1  & d_1+d_2-4 & 0\\
2 & k-2-\frac{(|V(G)|-2-(d_1+d_2-4))\alpha}{d_1+d_2-4} & \frac{(|V(G)|-2-(d_1+d_2-4))\alpha}{d_1+d_2-4}\\
0 & \alpha & k-\alpha
\end{array}
\right),
\end{eqnarray*}
that is,
\begin{eqnarray}\label{g4partion}
\widetilde{B}=\left(
\begin{array}{ccc}
1  & 4t & 0\\
2 & 4t-1-\frac{\alpha t}{2} & \frac{\alpha t}{2}\\
0 & \alpha & k-\alpha
\end{array}
\right)
\end{eqnarray}
with eigenvalues $k,\theta_1$ and $\theta_2$, where $\theta_1+\theta_2=4t-\frac{\alpha t}{2}-\alpha,\theta_1\theta_2=-4t-\frac{\alpha t}{2}+\alpha-1$. From Lemma \ref{quotienteigenvalues2} $(i)$, the eigenvalues of (\ref{g4partion}) interlace the eigenvalues of $A$, that is, $-3\leq\theta_1,\theta_2\leq2t-1$, and we obtain the following inequalities:
\begin{equation}\label{g4ineq1}
(-3)^2-(4t-\frac{\alpha t}{2}-\alpha)(-3)-4t-\frac{\alpha t}{2}+\alpha-1\ge 0,
\end{equation}
\begin{equation}\label{g4ineq2}
(2t-1)^2-(4t-\frac{\alpha t}{2}-\alpha)(2t-1)-4t-\frac{\alpha t}{2}+\alpha-1\ge 0.
\end{equation}

Inequalities (\ref{g4ineq1}) and (\ref{g4ineq2}) are only satisfied for $\alpha=4$, and for this value of $\alpha$, they become equalities. This means that (\ref{g4partion}) becomes
\begin{eqnarray}\label{g4partion2}
\widetilde{B}=\left(
\begin{array}{ccc}
1  & 4t & 0\\
2 & 2t-1 & 2t\\
0 & 4 & 4t-3
\end{array}
\right)
\end{eqnarray}
with eigenvalues $k,2t-1$ and $-3$. So we have tight interlacing and Lemma \ref{quotienteigenvalues2} $(ii)$ implies that this is an equitable partition.
\end{proof}

\subsection{Determining the order of the quasi-cliques for $\mathfrak{g}_5$}

In this subsection, we will determine the order of the quasi-cliques corresponding to an indecomposable factor isomorphic to $\mathfrak{g}_5$. For the rest of this subsection, we will assume that $\mathfrak{g}_5$ is an indecomposable factor of $\mathfrak{h}$ and that $\mathfrak{g}_5$ is as in Figure \ref{fg5'size}, where the slim vertex $x$ has fat neighbors $I_1,~I_2$
and $I_3$. Let $Q_{\mathfrak{h}}(I_j)$ be the quasi-clique corresponding to the fat vertex $I_j$ and $b_j=|V(Q_{\mathfrak{h}}(I_j))|$ for $j=1,2,3$. Without loss of generality, we may assume that $b_1\geq b_2\geq b_3$.
\begin{figure}[H]
\center
  \includegraphics[scale=1]{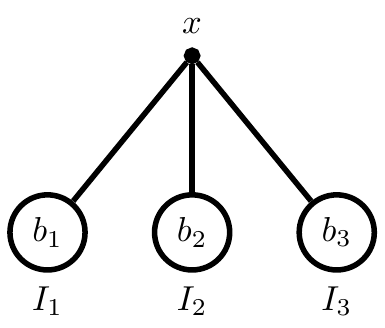}
  \vspace{0.1cm}
  \caption{$\mathfrak{g}_5$}
  \label{fg5'size}
\end{figure}

It is easy to see that $b_1-1+b_2-1+b_3-1=4t+1$, hence
\begin{equation}\label{g5firstcondition}
b_1+b_2+b_3=4t+4.
\end{equation}
Note that the above implies that there cannot be two quasi-cliques with order $2t+2$, so it follows that
\begin{equation}\label{g5secondcondition}
1\leq b_3\leq b_2\leq 2t+1.
\end{equation}

Let $e_G(x)$ be the number of edges in the subgraph of $G$ induced by the set of neighbors of $x$, $N_G(x)$. From (\ref{ways}) it follows:
\begin{equation}\label{g5'edges1}
e_G(x)=4t^2+2t=2\binom{2t+1}{2}.
\end{equation}
Now we give the following proposition to obtain bounds on $e_G(x)$ (of (\ref{g5'edges1})):

\begin{prp}\label{g5'prp1}
Let $1\leq i<j\leq 3$. Then any vertex $y~(y\neq x)$ in $Q_{\mathfrak{h}}(I_j)$ has at most $2$ neighbors in $V(Q_{\mathfrak{h}}(I_i))-\{x\}$.
\end{prp}
\begin{proof}We show it for $i=1$ and $j=2$. The other cases follow in a similar way. Suppose $y$ is a vertex in $Q_{\mathfrak{h}}(I_2)$ and $y\neq x$. Since $b_2=|V(Q_\mathfrak{h}(I_2))|\leq 2t+1$, and the indecomposable factors $\mathfrak{g}_3$ and $\mathfrak{g}_4$ do not have quasi-clique with order at most $2t+1$ (Lemma \ref{g3'} and Lemma \ref{g4'}), the indecomposable factor containing $y$ as slim vertex is isomorphic to $\mathfrak{g}_5$. Now we need the following claim:

\begin{claim}\label{minCommonNeighbors2} For a fat vertex $F\in V^f_{\mathfrak{h}}(y)$, we have $|N^s_{\mathfrak{h}}(I_1,F)|\leq 1$.
\end{claim}
\begin{proof} Clearly, when $F$ is the fat vertex $I_2$, the result holds. Suppose $F$ is a fat neighbor distinct from $I_2$. By Lemma \ref{commonneighborsg'} $(i)$, we have $|N^s_{\mathfrak{h}}(I_1,F)|\leq 2$. Now assume that $|N^s_{\mathfrak{h}}(I_1,F)|=2$ and  $N^s_{\mathfrak{h}}(I_1,F)=\{x^\prime, y^\prime\}$. By Lemma \ref{commonneighborsg'} $(i)$, it follows that the Hoffman subgraph induced by the slim vertices $x^\prime$ and $y^\prime$ and the fat vertices $I_1$ and $F$ is isomorphic to the indecomposable factor $\mathfrak{g_4}$ and by Lemma \ref{g4'}, we have $b_1=|V(Q_\mathfrak{h}(I_1))|=|V(Q_\mathfrak{h}(F))|=2t+2$. As $x\not\in V(Q_{\mathfrak{h}}(F))$ and $y\not\in V(Q_{\mathfrak{h}}(I_1))$, we obtain that $\{x^\prime, y^\prime\}\bigcap\{x,y\}=\emptyset$. By using Lemma \ref{g4'} again, we obtain that the partition $\{V_1,V_2,V_3\}=\big\{\{x^\prime,y^\prime\}, V(Q_{\mathfrak{h}}(I_1))\bigcup V(Q_{\mathfrak{h}}(F))-\{x^\prime,y^\prime\},V(G)-V(Q_{\mathfrak{h}}(I_1))\bigcup V(Q_{\mathfrak{h}}(F))\big\}$ is equitable and $x$ has exactly $2t-1$ neighbors in $V_2$, since $x\in V(Q_{\mathfrak{h}}(I_1))-\{x^\prime,y^\prime\}\subseteq V_2$. But, on the other hand, $x$ has at least $|(V(Q_{\mathfrak{h}}(I_1))-\{x^\prime,y^\prime\}-\{x\})\bigcup\{y\}|=2t$ neighbors in $V_2$. This gives a contradiction.
\end{proof}

We can finish now the proof of Proposition \ref{g5'prp1}.

Note that $y$ is the slim vertex of an indecomposable factor isomorphic to $\mathfrak{g}_5$, see Figure \ref{fg5'size1},
\begin{figure}[H]
\center
  \includegraphics[scale=0.8]{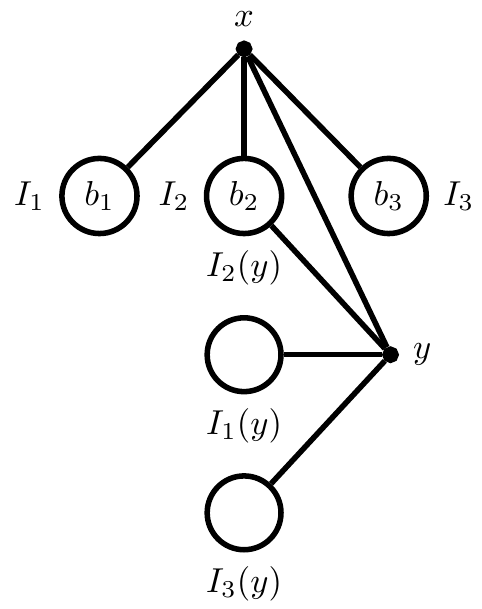}
  \vspace{0.1cm}
  \caption{}
  \label{fg5'size1}
\end{figure}
where $I_2(y)=I_2$.
Then from $N^s_{\mathfrak{h}}(I_1,I_2(y))=\{x\}$ (by Lemma \ref{commonneighborsg'} $(ii)$), $|N^s_{\mathfrak{h}}(I_1,I_1(y))|\leq 1$ and $|N^s_{\mathfrak{h}}(I_1,I_3(y))|\leq 1$, we find that $y$ has at most $2$ neighbors in $V(Q_{\mathfrak{h}}(I_1))-\{x\}$ and the result holds.
\end{proof}

From Proposition \ref{g5'prp1}, it follows
\begin{equation}\label{g5'edges2}
e_G(x)\leq \binom{b_1-1}{2}+\binom{b_2-1}{2}+\binom{b_3-1}{2}+ 2(b_2-1)+4(b_3-1).
\end{equation}

By using (\ref{g5'edges1}) and (\ref{g5'edges2}), we obtain
\begin{equation*}
\binom{b_1-1}{2}+\binom{b_2-1}{2}+\binom{b_3-1}{2}+ 2(b_2-1)+4(b_3-1)\geq 2\binom{2t+1}{2}.
\end{equation*}
This means
\begin{equation*}\begin{split}
2(2t+1)2t &\le (b_1-1)(b_1-2)+(b_2-1)(b_2-2)+(b_3-1)(b_3-2)+4(b_2-1)\\
                 &\quad+8(b_3-1)\\
              &= b_1^2-3b_1+b_2^2+b_3^2+4b_3+(b_2+b_3)-6\\
              &= b_1^2-3b_1+b_2^2+b_3^2+4b_3+(4t+4-b_1)-6\\
              &= (b_1-2)^2+b_2^2+(b_3+2)^2+4t-10,
\end{split}\end{equation*}
and we obtain $(b_1-2)^2+b_2^2+(b_3+2)^2\ge 8t^2+10$, where $1\leq b_3\leq b_2\leq b_1\leq 2t+2,$ and $b_3+b_2+b_1=4t+4.$

When $t>4$ holds, we find that $b_3\leq 2$ and there are three possible cases for the order of the quasi-cliques of $\mathfrak{g}_5$: $(b_1,b_2,b_3)=(2t+2,2t+1,1),$ $(b_1,b_2,b_3)=(2t+2,2t,2),$ or $(b_1,b_2,b_3)=(2t+1,2t+1,2).$

This shows the following lemma:
\begin{lema}\label{g5'}
Suppose that $\mathfrak{g}_5$ is an indecomposable factor of $\mathfrak{h}$ with fat vertices $I_1,I_2$ and $I_3$. For $i=1,2,3$, let $b_i$ be the order of the quasi-clique $Q_{\mathfrak{h}}(I_i)$ corresponding to the fat vertex $I_i$ in $\mathfrak{g}_5$ with $b_1\geq b_2\geq b_3$. If $t>4$, then one of the following holds:
\begin{enumerate}
\item [$(1)$] $(b_1,b_2,b_3)=(2t+2,2t+1,1);$

\item [$(2)$] $(b_1,b_2,b_3)=(2t+2,2t,2);$

\item [$(3)$] $(b_1,b_2,b_3)=(2t+1,2t+1,2).$
\end{enumerate}
\end{lema}

\section{Finishing the proof of Theorem \ref{maintheorem1intro}}\label{proof}

In Figure \ref{type}, we summarize what we have shown until now. We give the possible indecomposable factors together with the order of their quasi-cliques under the condition $t>4$. We will refer to a slim vertex $x$ having Type $i$ ($i=1,2,3,4,5$) if the indecomposable factor which contains $x$ is of Type $i$.
\begin{figure}[H]
\ffigbox{
\begin{subfloatrow}
\CenterFloatBoxes
\ffigbox[1.2\FBwidth]{\caption*{Type $1$: $\mathfrak{h^1}$}}{\includegraphics[scale=0.8]{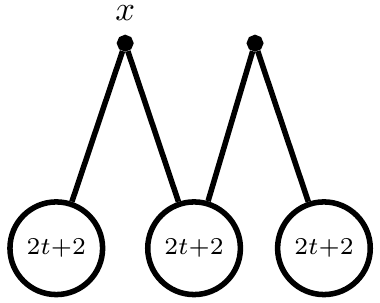}}
\ffigbox[1.2\FBwidth]{\caption*{Type $2$: $\mathfrak{h^2}$}}{\includegraphics[scale=0.8]{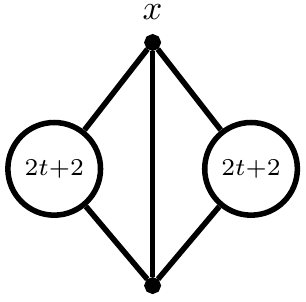}}
\ffigbox[1.2\FBwidth]{\caption*{Type $3$: $\mathfrak{h^3}$}}{\includegraphics[scale=0.8]{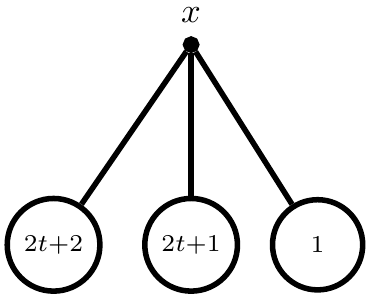}}
\end{subfloatrow}

\vspace{0.2cm}
\begin{subfloatrow}
\CenterFloatBoxes
\ffigbox[1.2\FBwidth]{\caption*{Type $4$: $\mathfrak{h^4}$}}{\includegraphics[scale=0.8]{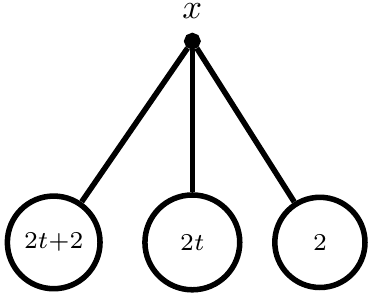}}
\ffigbox[1.2\FBwidth]{\caption*{Type $5$: $\mathfrak{h^5}$}}{\includegraphics[scale=0.8]{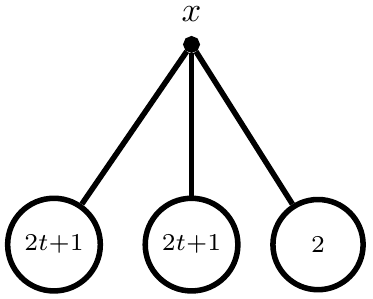}}
\end{subfloatrow}
}
{\caption{}\label{type}}
\end{figure}

Suppose that there are $n_i$ vertices of Type $i$ and $q_j$ quasi-cliques of order $j$ in $G$, where $i=1, 2, 3, 4, 5$ and $j=2t, 2t+1, 2t+2$. Consider the sets $W_i=\big\{(x,Q) \mid x\in V(Q)$, where $Q$ is a quasi-clique of order $2t-1+i$ corresponding to some fat vertex in $\mathfrak{h}\big\}$, $i=1,2,3$. Then, by double counting the cardinalities of the sets $W_1$, $W_2$ and $W_3$, we obtain

\begin{gather}
n_4=2tq_{2t},\label{w1}\\
n_3+2n_5=(2t+1)q_{2t+1},\label{w2}\\
2n_1+2n_2+n_3+n_4=(2t+2)q_{2t+2},\label{w3}\\
n_1+n_2+n_3+n_4+n_5=|V(G)|=2(t+1)^2.\label{v}
\end{gather}

From (\ref{w1}),(\ref{w2}),(\ref{w3}) and (\ref{v}), we obtain

\begin{equation}\label{eq1}
2tq_{2t}+(2t+1)q_{2t+1}+(2t+2)q_{2t+2}=(2t+2)^2,
\end{equation}

which implies
\begin{equation}\label{eq2}
-q_{2t}+q_{2t+2}\equiv 1\pmod{2t+1}.
\end{equation}

From (\ref{eq1}), it is easy to see that
\begin{equation}\label{q2t}
q_{2t+2}\leq 2t+2.
\end{equation}
From (\ref{w1}) and (\ref{w3}), it follows that $n_4=2tq_{2t}\leq (2t+2)q_{2t+2}$, hence $q_{2t}\leq \lfloor(1+1/t)q_{2t+2}\rfloor\leq q_{2t+2}+2$. This shows that the only possible solutions of (\ref{eq2}) are the following:

\begin{description}
\item[Case 1:] $q_{2t+2}=q_{2t}+2t+1+1$.

By (\ref{eq1}) and (\ref{q2t}), it follows that $q_{2t+2}=2t+2$ and $q_{2t}=q_{2t+1}=0$.\\

\item[Case 2:] $q_{2t+2}=q_{2t}+1$.

In this case (\ref{eq1}) becomes
$$2q_{2t}+q_{2t+1}=2t+2$$
So $q_{2t}\leq t+1$. If there exists a quasi-clique $Q$ with order $2t$, then every vertex in this quasi-clique belongs to Type $4$ and we can find a quasi-clique $Q^\prime$ with order $2t+2$ exactly sharing this vertex with $Q$ by Lemma \ref{commonneighborsg'} $(i)$. This means that $q_{2t+2}\geq |V(Q)|= 2t$. So $t+1\geq q_{2t}= q_{2t+2}-1\geq 2t-1$, but this is not possible. Hence $q_{2t}=0$, and this implies $q_{2t+2}=1$ and $q_{2t+1}=2t+2$.
\end{description}
Summarizing, we only have the following two cases:\\
{\bf Case 1}: $q_{2t}=0,~q_{2t+1}=0,~q_{2t+2}=2t+2$;\\
{\bf Case 2}: $q_{2t}=0,~q_{2t+1}=2t+2,~q_{2t+2}=1$.\\

Now we are going to determine the $n_i$'s for $i=1,2,3,4,5$. Observe that $q_{2t}=0$ holds for both cases, which implies that $n_4=0$ holds in both cases by using (\ref{w1}).

\begin{prp}\label{case1}
If $q_{2t}=q_{2t+1}=0,~q_{2t+2}=2t+2$ and $t>4$, then $G$ is the $2$-clique extension of the $(t+1)\times (t+1)$-grid.
\end{prp}
\begin{proof}Since $q_{2t+1}=0$, we find $n_3=n_5=0$ from (\ref{w2}). Hence all vertices of $G$ are of Type $1$ or Type $2$ and every vertex of $G$ has exactly two fat neighbors. We want to show that $n_1=0$. Suppose this is not the case. Then there exists a vertex $x$ belonging to Type $1$ and the Hoffman graph shown in Figure \ref{fmatheorem1} is an indecomposable factor of $\mathfrak{h}$, where $x,x^\prime\in N_\mathfrak{h}^s(F_2)$ and $x\not\sim x^\prime$.
\begin{figure}[H]
\center
  \includegraphics[scale=1]{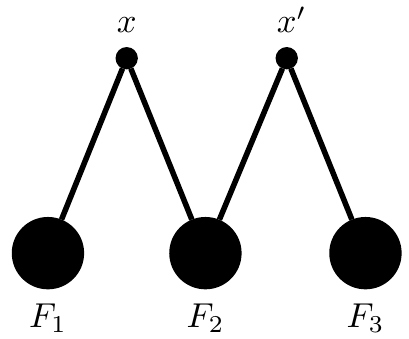}
  \vspace{0.1cm}
  \caption{}
  \label{fmatheorem1}
\end{figure}

In a similar way as in Claim \ref{minCommonNeighbors2}, we can show that, for any neighbor $y$ of $x$ in the quasi-clique $Q_{\mathfrak{h}}(F_2)$ and for any fat vertex $F\in V^f_{\mathfrak{h}}(y)$, it follows that $|N^s_{\mathfrak{h}}(F_1,F)|\leq1$. Observing that $y$ has only one fat neighbor besides the fat vertex $F_2$, this implies that $y$ has at most one neighbor in $Q_{\mathfrak{h}}(F_1)$ besides $x$. Suppose that $a_1=|V(Q_{\mathfrak{h}}(F_1))|,~a_2=|V(Q_{\mathfrak{h}}(F_2))|$. Since $x^\prime$ has no neighbor in the quasi-clique $Q_{\mathfrak{h}}(F_1)$, it implies that $Q_{\mathfrak{h}}(F_1)$ cannot be a clique by Lemma \ref{HoffmanBound}. Therefore, the subgraph of $G$ induced by $V(Q_{\mathfrak{h}}(F_1))-\{x\}$ is not a clique. By counting the number of triangles throught $x$ we obtain
\begin{equation*}\begin{split}
A_{(x,x)}^3&\leq2\left(\binom{a_1-1}{2}-1\right)+2\binom{a_2-2}{2}+2(a_2-2)\\
           &\leq2\left(\binom{2t+1}{2}-1\right)+2\binom{2t}{2}+2\cdot2t\\
           &=8t^2+4t-2.
\end{split}
\end{equation*}
But, as $G$ has the same spectrum as the $2$-clique extension of the $(t+1)\times (t+1)$-grid, we obtain that $A_{(x,x)}^3=8t^2+4t$ by (\ref{ways}). This gives a contradiction. Hence, we just showed that all the vertices of $G$ are of Type $2$.\\

Now, consider the following equivalence relation $\mathcal{R}$ on the vertex set $V(G)$:

$$ x\mathcal{R}x^\prime{\text{ if and only if }}\{x\}\cup N(x)=\{x^\prime\}\cup N(x^\prime), {\text{ where }} x,x^\prime\in V(G).$$
It means that for each vertex $x$, there exists an unique distinct vertex $x^\prime$ such that $x\mathcal{R}x^\prime$ and $x^\prime\sim x$. So two vertices in the same equivalent class induce a $2$-clique. Let us define a graph $\underline{G}$ whose vertices are the equivalent classes, and such that two classes $\{x,x^\prime\}$ and $\{y,y^\prime\}$ are adjacent in $\underline{G}$ if and only if $x\sim y,x^\prime\sim y,x\sim y^\prime,x^\prime\sim y^\prime.$ Then $\underline{G}$ is a regular graph with valency $2t$, and $G$ is the $2$-clique extension of $\underline{G}$. Note that the spectrum of $\underline{G}$ follows immediately from (\ref{spectrum1}) and (\ref{spectrum2}) and is equal to
$$\big\{(2t)^{1}, (t-1)^{2t}, (-2)^{t^{2}}\big\}.$$

Since $\underline{G}$ is a connected regular graph with valency $2t$ with multiplicity $1$, and since it has exactly three distinct eigenvalues, it follows that $\underline{G}$ is a strongly regular graph with parameters $\big((t+1)^2, 2t, t-1, 2\big)$. From \cite{Sh1959}, it follows that if $t\neq 3$, then the graph with these parameters is unique and is the $(t+1)\times(t+1)$-grid. So we obtained that $G$ is the $2$-clique extension of the $(t+1)\times(t+1)$-grid when $t>4$.
\end{proof}

Now let us assume that we are in {\bf Case 2}, that is $q_{2t}=0,~q_{2t+1}=2t+2$, and $q_{2t+2}=1$. We have already seen that $n_4=0$. We will show that this case is impossible. But to show this, we will need a few lemmas.

As a vertex of Type $1$ or Type $2$ lies in two distinct quasi-cliques of order $2t+2$ and $q_{2t+2}=1$, we find that there are no vertices of Type $1$ or Type $2$. So we obtain $n_1=n_2=0$. This implies $n_3=2t+2$ and $n_5=2t(t+1)$ by (\ref{w3}) and (\ref{v}). As $n_1=0$, all quasi-cliques are actually cliques since every vertex is adjacent to all of the vertices in the same quasi-clique except itself.

Let $Q$ be the unique quasi-clique of order $2t+2$ and let $\mathcal{L}=\{Q^\prime\mid Q^\prime$ is a quasi-clique of order $2t+1\}$. We already noticed that $Q$ and $Q^\prime\in\mathcal{L}$ are actually cliques. Now we will show the following lemma:
\begin{lema}\label{sum1}
\item[($i$)] For every vertex $x$ in $Q$, there exists an unique quasi-clique $Q^\prime_x\in\mathcal{L}$ such that $x\in V(Q^\prime_x)$;

\item[($ii$)] For distinct vertices $x_1$ and $x_2$ in $Q$, the quasi-cliques $Q^\prime_{x_1}$ and $Q^\prime_{x_2}$ are distinct;

\item[($iii$)] For every quasi-clique $Q^\prime\in\mathcal{L}$, there exists an unique vertex $x_{Q^\prime}$ such that $x_{Q^\prime}\in V(Q)$;

\item[($iv$)] For distinct quasi-cliques $Q_1^\prime$ and $Q_2^\prime$ in $\mathcal{L}$, the vertices $x_{Q_1^\prime}$ and $x_{Q_2^\prime}$ are distinct.
\end{lema}
\begin{proof}
$(i)$ It follows from before the fact that, for all $x\in V(Q)$, $x$ is of Type $3$.
\vspace{0.2cm}

$(ii)$ By Lemma \ref{commonneighborsg'} $(ii)$, we have $|V(Q^\prime)\bigcap V(Q)|\leq 1$ for any $Q^\prime\in\mathcal{L}$. If $Q^\prime_{x_1}$ and $Q^\prime_{x_2}$ are the same, then $Q^\prime_{x_1}$ shares two common vertices with $Q$, it is not possible. So the result follows.
\vspace{0.2cm}

$(iii)$ Since $|\mathcal{L}|=q_{2t+1}=2t+2$ and $|V(Q)|=2t+2$, it follows from $(i)$ and $(ii)$.

$(iv)$ It follows from $(i)$-$(iii)$.
\end{proof}

Let $W=V(G)-V(Q)$ and let $G^\prime$ be the induced subgraph of $G$ on $W$. Let $G^{\prime\prime}$ be the spanning subgraph of $G^\prime$ such that the vertices $w_1,~w_2$ are adjacent in $G^{\prime\prime}$ if there exists a quasi-clique $Q^\prime\in\mathcal{L}$ such that $w_1$ and $w_2$ are in $Q^\prime$. Now we have the following lemma:

\begin{lema}\label{ismorphicgraph}
The graph $G^{\prime\prime}$ is the line graph of the \emph{cocktail-party graph} $CP(2t+2)$.
\end{lema}

\begin{proof}
Define the graph $H$ with vertex set $\mathcal{L}$ and two quasi-clique $Q_1^\prime,Q_2^\prime\in\mathcal{L}$ are adjacent if they intersect in a unique element. It is easy to see that the graph $G^{\prime\prime}$ is the line graph of $H$. As any quasi-clique $Q^\prime$ of $\mathcal{L}$ has $2t$ vertices in $W$ and any vertex in $W$ lies in two quasi-cliques in $\mathcal{L}$, it follows that $H$ is $2t$-regular. So $H$ is the cocktail-party graph $CP(2t+2)$ as it has $2t+2$ vertices. Hence, the lemma holds.
\end{proof}

Let $\Omega=\{x_1, \ldots,x_{t+1},x_1^\prime, x_2^\prime, \ldots, x_{t+1}^\prime\}$, and let $\Omega^2=\{2\text{-subsets of }\Omega\}-\bigcup_{i=1}^{t+1}\{x_i,x_i^\prime\}$. (For convenience, we will use $x_ix_j$ to represent the subset $\{x_i,x_j\}$, and similarly for the other $2$-subsets in $\Omega^2$.) We define the graph $G^0$ with vertex set $\Omega\bigcup\Omega^2$ and three kinds of edges as follows:
\begin{enumerate}
\item[$(1)$] the edges of the form $\{x,y\}$, where $x,y\in\Omega$;

\item[$(2)$] the edges of the form $\{x,xy\}$, where $x\in\Omega,xy\in\Omega^2$;

\item[$(3)$] the edges of the form $\{xy,xz\}$, where $xy,xz\in\Omega^2$.
\end{enumerate}
By Lemma \ref{ismorphicgraph}, Lemma \ref{sum1} and the definition of $Q$, we see that $G^0$ is isomorphic to a spanning subgraph of $G$, and hence we can identify $V(G)$ with $\Omega\bigcup\Omega^2$.

Now consider the partition $\pi=\{V_1,V_2,V_3,V_4\}$ of $V(G)$, where
\begin{gather*}
V_1=\{x_1,x_1^\prime\},\\
V_2=\{x_i,x_i^\prime: 2\leq i\leq t+1\},\\
V_3=\{x_1x_i,x_1x_i^\prime,x_1^{\prime}x_i, x_1^\prime x_i^\prime:2\leq i\leq t+1\},\\
V_4=\{x_ix_j,x_ix_j^\prime,x_i^\prime x_j,x_i^\prime x_j^\prime,2\le i<j\le t+1\}.
\end{gather*}
The quotient matrix $\widetilde{B}$ of the adjacency matrix $A$ of $G$ with respect to the above partition $\pi$ is given as follows:
\begin{eqnarray}\label{partitionmatrix}
\widetilde{B}=\left(
\begin{array}{cccc}
1&2t& 2t &0  \\
2&2t-1&2&2t-2 \\
1&1&\alpha&4t-1-\alpha\\
0&2&\frac{2(4t-1-\alpha)}{t-1}&4t-1-\frac{2(4t-1-\alpha)}{t-1}
\end{array}
\right)
\end{eqnarray}
with $2t\leq\alpha\leq 2t+1$.

We will show that $\alpha=2t+1$, and hence the partition $\pi$ is an equitable partition of $G$.

To show this, note that by (\ref{ways}), we have
\begin{equation}\label{x11}\begin{split}
A_{(x_1,x_1^\prime)}^3&=24t+1-(5-2t)\lambda_{x_1,x_1^\prime}\\
              &=24t+1-(5-2t)\cdot2t\\
              &=4t^2+14t+1.
\end{split}\end{equation}

On the other hand,
\begin{equation}\label{x12}
A_{(x_1,x_1^\prime)}^3=4t+1+\sum\nolimits_{z\in G_1(x_1)\bigcap G_1(x_1^\prime)}\lambda_{x_1,z}+\sum\nolimits_{z\in G_2(x_1)\bigcap G_1(x_1^\prime)}\mu_{x_1,z},
\end{equation}
where $|G_1(x_1)\bigcap G_1(x_1^\prime)|=2t,~|G_2(x_1)\bigcap G_1(x_1^\prime)|=2t$ and
\begin{equation*}\begin{split}
\lambda_{x_1,z}=2t+1,&\text{ for } z\in G_1(x_1)\bigcap G_1(x_1^\prime),\\
3\leq\mu_{x_1,z}\leq 4,&\text{ for } z\in G_2(x_1)\bigcap G_1(x_1^\prime).
\end{split}
\end{equation*}

Then, from (\ref{x11}) and (\ref{x12}), we obtain that
$$\mu_{x_1,z}=4, \text{for }z\in G_2(x_1)\bigcap G_1(x_1^\prime),$$
which implies that $\alpha=2t+1$. Therefore, we have an equitable partition with partition diagram as shown in Figure \ref{partition}.
\begin{figure}[H]
\center
  \includegraphics[scale=0.8]{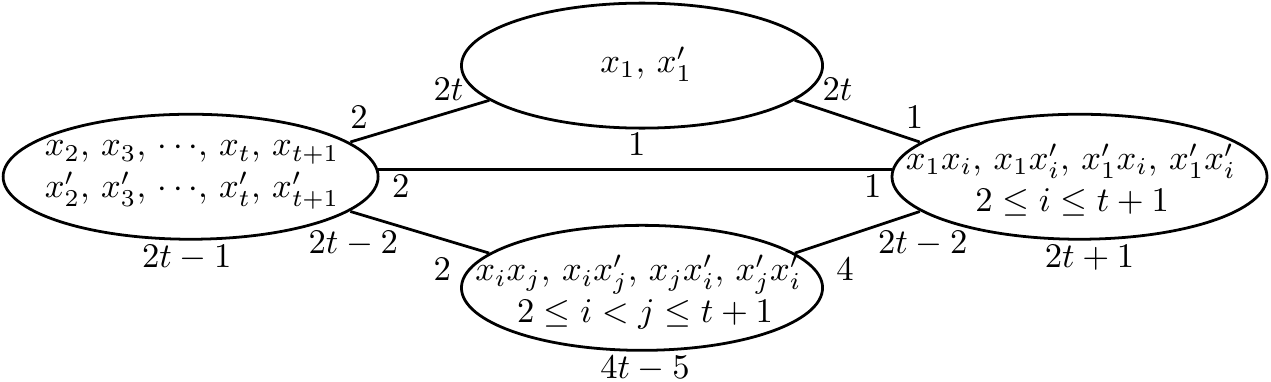}
  \caption{Equitable partition}
  \label{partition}
\end{figure}

In this case, the quotient matrix (\ref{partitionmatrix}) becomes
\begin{equation}
\widetilde{B}=\left(
  \begin{array}{cccc}
    1 & 2t & 2t & 0 \\
    2 & 2t-1 & 2 & 2t-2 \\
    1 & 1 & 2t+1 & 2t-2 \\
    0 & 2 & 4 & 4t-5
  \end{array}
\right)
\end{equation}
with eigenvalues $\big\{4t+1,2t-1,t-2\pm\sqrt{t^2-1}\big\}$.

From Lemma \ref{quotienteigenvalues1}, we find that the eigenvalues of $\widetilde{B}$ should be the eigenvalues of $A$. But $B$ has eigenvalues $t-2\pm\sqrt{t^2-1}$, which are not the eigenvalues of $A$. So we obtain a contradiction. This shows that the case $q_{2t}=0,~q_{2t+1}=2t+2,~q_{2t+2}=1$ is not possible. This concludes the proof to show that $G$ is the $2$-clique extension of the $(t+1)\times(t+1)$ grid.

%

\begin{remark}
Note that we used walk-regularity (which follows from the fact that the $2$-clique extension of the $(t+1)\times(t+1)$-grid is regular with exactly $4$ distinct eigenvalues) to show this result, and therefore it is not so clear how to extend this result to the $2$-clique extension of a non-square grid graph.
\end{remark}


\begin{thebibliography}{}
\bibitem{ABH2015} A. Abiad, A.E. Brouwer, and W.H. Haemers, Godsil-McKay switching and isomorphism, \emph{Electr. J. Linear Algebra} \textbf{28} (2015), 4--11.

\bibitem{BDK2008} S. Bang, E.R. van Dam, and J.H. Koolen, Spectral characterization of the Hamming
graphs, \emph{Linear Algebra Appl.} \textbf{429} (2008), 2678--2686.

\bibitem{BCN} A.E. Brouwer, A.M. Cohen, and A. Neumaier, Distance-Regular Graphs, Springer-Verlag,
Berlin-New York, 1989.

\bibitem{D1995} E.R. van Dam, Regular graphs with four eigenvalues, \emph{Linear Algebra Appl.} \textbf{226-228} (1995), 139--162.

\bibitem{D1996} E.R. van Dam, Graphs with Few Eigenvalues - an Interplay between Combinatorics and Algebra,
Ph.D. thesis, 1996.


\bibitem{DH2003} E.R. van Dam and W.H. Haemers, Which graphs are determined by their spectrum?, \emph{Linear Algebra Appl.} \textbf{373} (2003), 241--272.

\bibitem{VKT} E.R. van Dam, J.H. Koolen, and H. Tanaka, Distance-regular graphs, \emph{Electron. J. Combin.}
(2016), $\sharp$DS22.

\bibitem{GK2016} A. Gavriliouk and J.H. Koolen, On a characterization of the Grassmann graphs over the binary field, preprint (2016).

\bibitem{Ch&G} C. Godsil and G. Royle, \emph{Algebraic Graph Theory}, Springer-Verlag, New York, 2001.

\bibitem{GKMST2015} G. Greaves, J. Koolen, A. Munemasa, Y. Sano, and T. Taniguchi, Edge-signed graphs with smallest eigenvalue greater than $-2$, \emph{J. Combin. Theory Ser. B} \textbf{110} (2015), 90--111.

\bibitem{H1995} W.H. Haemers, Interlacing eigenvalues and graphs, \emph{Linear Algebra Appl.} {\bf 226-228} (1995), 593--616.

\bibitem{H1963} A.J. Hoffman, On the polynomial of a graph, \emph{Amer. Math. Monthly} \textbf{70} (1963), 30--36.

\bibitem{JKMT2014} H.J. Jang, J. Koolen, A. Munemasa, and T. Taniguchi, On fat Hoffman graphs with smallest eigenvalue
at least -3, {\em ARS Mathematica Contemporanea} {\bf 7} (2014), 105--121.

\bibitem{KYY} J.H. Koolen, J.Y. Yang and Q.Q. Yang, A generalization of a theorem of Hoffman, submitted.

\bibitem{Sh1959} S.S. Shrikhande, The uniqueness of the $L_2$ association scheme, \emph{Ann. Math. Statist.} \textbf{30} (1959), 781--798.

\bibitem{WN1995} R. Woo and A. Neumaier, On graphs whose smallest eigenvalue is at least $-1-\sqrt{2}$, \emph{Linear
Algebra Appl.} \textbf{226-228} (1995), 577--591.

\bibitem{Y2012} H. Yu, On the limit points of the smallest eigenvalues of regular graphs, \emph{Des. Codes Cryptography} \textbf{65} (2012), 77--88.



\end{thebibliography}
\end{document}